\documentclass{amsart}
\numberwithin{equation}{section}

\def\epsilon{\varepsilon}
\def\euler{\chi}

\def\Pf{{\rm{Pf}}}

\def\parti#1#2{\frac{\partial #1 } {\partial #2} }

\def\upto{\nearrow}

\def\Sc{{\rm R}}
\def\Ricci{{\rm{Ricci}}}
\def\Rc{{\rm{Rc}}}
\def\ti{\tilde}

\def\beq{ \begin{equation}  \begin{split} }
\def\eeq{ \end{split}  \end{equation} }

\def\lap{\Delta}

\def\partt{\frac{\partial }{\partial t} }

\def\phi{\varphi}
\def\R{\mathbb R}

\def\M^n#1#2{\mathcal M^{#1}\left(#2\right)}

\def\ddt{\frac{d}{d t}}
\def\dds{\frac{d}{d s}}

\def\part{\partial}

\def\grad{\nabla}

\DeclareMathOperator{\vol}{vol}

\DeclareMathOperator{\Riem}{Riem}
\DeclareMathOperator{\Weil}{Weil}
\DeclareMathOperator{\Rm}{Rm}

\newtheorem{theo}{Theorem}[section]
\newtheorem{lemma}[theo]{Lemma}

\newtheorem{coro}[theo]{Corollary}

\theoremstyle{definition}

\theoremstyle{remark}
\newtheorem{remark}[theo]{Remark}

\begin{document}
\title[Ricci flow integral curvature estimates]{Some integral curvature estimates for the Ricci flow in four
  dimensions}
\thanks{We gratefully acknowledge the support of SFB TR71 of the DFG (German Research Foundation) and 
the University of Freiburg, where a part of this work was carried out.}

\author{Miles Simon}
\address{Miles Simon: Otto von Guericke University, Magdeburg, IAN,
  Universit\"atsplatz 2, Magdeburg 39104, Germany}
\curraddr{}
\email{miles point simon at  ovgu point de}

%    General info
% 53C44 Geometric evolution equations (mean curvature flow)
% 35Bxx Qualitative properties of solutions
% 35B35 Stability, boundedness
\date{\today}

\dedicatory{}


\begin{abstract}
We consider solutions $(M^4,g(t)),  0 \leq t <T$,  to Ricci flow on
compact, four dimensional manifolds without boundary.
We prove integral curvature estimates which are valid for any such solution.
In the case that the scalar curvature is bounded and $T< \infty$, we show that
these estimates imply that the (spatial) integral of the square of the norm of the Riemannian curvature is bounded by
a constant independent of time $t$ for all $0 \leq t<T$ and that the space time
integral over $M \times [0,T)$ of the fourth power of the norm of the Ricci curvature is
bounded.
\end{abstract}

\maketitle
\section{Introduction}

We consider arbitrary smooth solutions to
Ricci flow, $\partt g(t) = -2\Rc(g(t))$ for all $t \in [0,T)$, on
closed, four manifolds without boundary. We assume that the scalar curvature
satisfies $\Sc(\cdot,0) >-1$ at time zero. If this is not the case, then we can always
scale the solution by a constant to obtain  a new solution satisfying
this inequality. The Ricci flow was first introduced and studied by R. Hamilton
in  \cite{HaThree}. We show that the following (and
other) integral estimates hold.
\begin{theo}\label{intestgen}
Let $(M^4,g(t))_{t\in [0,T)}$ be a smooth solution to Ricci flow  on a
compact four dimensional manifold $M^4$ without boundary and assume
that the scalar curvature satisfies $\inf_M
\Sc(\cdot,0) > -1$ at time zero.  Then

\begin{eqnarray}
&& \int_M \frac{|\Rc|^2(\cdot,S)}{(\Sc(\cdot,S) +2)} d\mu_{g(S)}  +
 \int_0^S  \int_M  \frac{|\Rc|^4(\cdot,t)}{(\Sc(\cdot,t) +2)^2} d\mu_{g(t)} dt  \cr 
&& \ \ \ \ \leq 
2^2\pi^2 \chi(e^{64S} - 1) + e^{64S}\int_M
\frac{|\Rc|^2(\cdot,0)}{(\Sc(\cdot,0) +2)} d\mu_{g(0)} \cr
&& \ \ \ \ \ \
+2^{10}e^{64S}\int_0^S \int_M \Sc^2(\cdot,t) d\mu_{g(t)} dt   \cr
&& =: c_0(M,g(0),S) + 2^{10}e^{64S}\int_0^S \int_M \Sc^2(\cdot,t) d\mu_{g(t)} dt \label{generalintro1}
 \end{eqnarray}
and
\begin{eqnarray}
 &&\int_M |\Rc|(\cdot,S) d\mu_{g(S)} \cr
&& \ \ \ \ \leq  \vol(M,g(S)) + 2c_0(M,g(0),S)  + 2^{11}  e^{64S}\int_0^S \int_M \Sc^2(\cdot,t) d\mu_{g(t)} dt  \label{generalintro2}
\end{eqnarray}
and
\begin{eqnarray}
&&\int_0^S \int_M |\Rc|^2  d\mu_{g(t)} dt \cr
&&\ \ \ \ \leq \int_0^S \vol(M,g(t)) dt
+2^3c_0(M,g(0),S)
+ 2^{13} e^{64S}\int_0^S \int_M \Sc^2(\cdot,t) d\mu_{g(t)} dt
 \label{generalintro3}
\end{eqnarray}
and
\begin{eqnarray}
&&\int_0^S \int_M |\Rm|^2  d\mu_{g(t)} dt \cr 
&& \ \  \
\leq 4\int_0^S \vol(M,g(t)) dt 
+ 2^5( c_0(M,g(0),S) + \pi^2 \chi S) \cr
&& \ \ \ \ +2^{15} e^{64S}\int_0^S \int_M \Sc^2(\cdot,t) d\mu_{g(t)} dt 
 \label{generalintro4} 
\end{eqnarray}
for all $0\leq S <T$, where $\chi = \chi(M)$ is the
Euler-characteristic of $M$, and 
\begin{eqnarray}
&& c_0(M,g(0),S):= 2^2\pi^2 \chi(e^{64S} - 1) + e^{64S}\int_M
\frac{|\Rc|^2(\cdot,0)}{(\Sc(\cdot,0) +2)} d\mu_{g(0)} 
\end{eqnarray}
is defined in the statement above.
\end{theo}

In the case that the scalar curvature is positive everywhere,
a minor modification of the proof of these estimates leads to
similar estimates which don't contain volume terms: see Theorem \ref{posscalar}.

In the remainder of the paper we consider the special case that $T< \infty$
and that the solution has bounded scalar curvature, $\sup_{M^4\times
  [0,T)} |\Sc|\leq 1$. In this case we see, that the above estimates
(and the proofs thereof) imply
\begin{eqnarray}
&&\sup_{ t\in [0,T)} \int_M |\Riem(\cdot,t)|^2 d\mu_{g(t)} \leq c_1 < \infty \cr
&&\int_0^T\int_M |\Rc|^4(\cdot,t)d\mu_{g(t)} dt\leq c_2 < \infty
\end{eqnarray}
for explicit  constants $c_1=c_1(M,g(0),T)$ and $c_2(M,g(0),T)$  (see Theorem \ref{integralest}).

In another paper, \cite{BZ}, which recently appeared, the authors 
also consider Ricci flow of four manifolds with bounded scalar curvature.
Using different methods, they also independently showed, in this case, that the $L^2$ norm of the
Riemannian curvature  remains bounded as $t \upto T$, and they investigate the structure
of the limiting space one obtains by letting  $t \upto
T$: see Theorem 1.8 and Corollary 1.11 of \cite{BZ}.
In a sequel paper, \cite{Si}, we examine the structure of (possibly) singular limits $(X,d_X) := \lim_{t \upto T}
(M,d(g(t)))$, where the limit is a Gromov-Hausdorff limit. This limit
always exists, and we show that $(X,d_X)$ is a $C^0$-Riemannian orbifold, and that it is
possible to flow this space for a short time using the orbifold Ricci flow.

\section{Background, previous results and notation}
Here we list some integral curvature estimates that have been shown to
hold for the Ricci flow in a general setting. This is by no means an exhaustive list.
For more references, we refer to references in the papers we have
listed here.

In the paper, \cite{HaSurface},  the author showed that\\
$\int_M (\Sc  \log\Sc)(\cdot,t) d\mu_{g(t)}  \leq  \int_M (\Sc
\log\Sc)(\cdot,0) d\mu_{g(0)}$ for all $t>0$, 
for any solution to the normalised Ricci flow $\partt g = -2\Rc  + rg$
on a surface (that is, on a two dimensional manifold) which has $\Sc(\cdot,0) >0$, where $r(t):= \frac{\int_M
  \Sc(\cdot,t) d\mu_{g(t)}}{\vol(M,g(t))}$ (see Theorem 7.2 in \cite{HaSurface}).
In \cite{ChowI}, the author showed that\\
 $\int_0^{\infty} \int_M (\Sc(\cdot,t) -r(t))^2
d\mu_{g(t)} dt < \infty$ for any solution to Ricci flow on the sphere.

In the paper \cite{TZ}, the authors proved 
integral curvature estimates for
solutions to the normalised K\"ahler Ricci flow on compact manifolds with bounded
diameter and positive scalar curvature (these solutions exist for all
time and have bounded scalar curvature and diameter due to a result of
Perelman: see \cite{ST}). In particular, they show
there, that \\
$\int_{M^n} (|\Riem|^2 + |\Ricci|^4)(\cdot,t) d\mu_{g(t)} \leq \Lambda $ for all $t>0$ for some
$\Lambda< \infty$. This estimate is obtained by  showing that
various integral quantities containing derivatives of  the potential function
remain bounded as time increases, and using then the Chern-Weil Theory (see
Section 4 of the paper \cite{TZ} for details, in particular Lemma 4.2 and
Theorem 4.5 there).

As we mentioned in the introduction, in the paper \cite{BZ} the
authors  independently
recently showed, that if the scalar curvature is bounded on $[0,T)$
and $M^n
= M^4$ is a four dimensional smooth closed manifold, then the $L^2$
integral of the Riemannian curvature remains bounded as $t \upto T$, 
and  they investigate the structure
of the limiting space one obtains by letting  $t \upto
T$: see Theorem 1.8 and Corollary 1.11 in \cite{BZ}.

\hfill\break
{\bf Notation}:
\hfill\break
We use the Einstein  summation convention, and
we use the notation of Hamilton \cite{HaThree}.\hfill\break
For $i \in \{1,\ldots,n\}$,
$ \parti{}{x^i}$ denotes a coordinate vector, and $dx^i$  is the corresponding one
form.\hfill\break
$(M^n,g)$ is an $n$-dimensional  Riemannian manifold with Riemannian metric $g$.
\hfill\break
$g_{ij} = g(\parti{}{x^i},\parti{}{x^i}) $ is the Riemannian metric $g$
with respect to this coordinate system.\hfill\break
$g^{ij}$ is the inverse of the Riemannian metric ($g^{ij}g_{ik} = \delta_{jk}$).\hfill\break
$d\mu_{g}$ is the volume form associated to $g$.\hfill\break
$\Rm(g)_{ijkl} = {{}^g\Riem}_{ijkl} = \Riem(g)_{ijkl} = \Sc_{ijkl}$ is the full
Riemannian curvature Tensor.\hfill\break
$\Weil(g)_{ijkl}$ is the Weil Tensor.\hfill\break
${{}^g\Rc}_{ij} = \Ricci_{ij} = \Sc_{ij}:= g^{kl}\Sc_{ikjl}$ is the Ricci curvature.\hfill\break
$\Sc:= \Sc_{ijkl}g^{ik}g^{jl}$ is the scalar curvature. \hfill\break
${}^{g}\grad T = \grad T$ is the covariant derivative of $T$ with respect to $g$. For example,
locally $\grad_i T_{jk}^s = (\grad T)( \parti{}{x^i}, 
\parti{}{x^j}, \parti{}{x^k}, d x^s)$
(the first index denotes the direction in which the covariant derivative
is taken) if locally $T = T_{jk}^s dx^j \otimes dx^k \otimes \parti{}{x^s}$. \hfill\break
$|T| = {{}^g|T|}$ 
is the norm of a tensor with respect to a metric $g$. For example
for $T = T_{jk}^s dx^j \otimes dx^k \otimes \parti{}{x^s}$.
$|T|^2 = g^{im}g^{jn}g_{ks} T_{ij}^s T_{mn}^k$.\hfill\break
Sometimes we  make it clearer which Riemannian metric we are considering by
including the metric in the definition.
For example $\Sc(h)$ refers to the scalar curvature of the Riemannian
metric $h$.\hfill\break
We suppress the $g$ in the notation used for the norm, $|T| = {}^g|T| $,
and for other quantities,
in the case that is is clear from the context which Riemannian metric
we are considering.\\

\section{Integral Inequalities in four dimensions}

In this paper we  consider (unless otherwise stated) smooth
families of Riemannian metrics 
$(M^n,g(t))_{t \in [0,T)}$ on $n$ dimensional compact, connected  manifolds
without boundary
which solve the Ricci flow equation
$$\partt g(t) = -2\Ricci(g(t)),$$  for all $t \in [0,T)$. 
We will mainly be interested in the case that $n=4$.

The following evolution equations hold for the Ricci flow
(see \cite{HaThree})
\begin{equation}
\begin{split}
& \partt |\Rc|^2  = \lap |\Rc|^2 -2| \grad \Rc|^2 + 
4\Rm^{ikjl} \Rc_{ij} \Rc_{kl}  \\ \label{riccievn}
& \partt \Sc  = \lap \Sc + 2|\Rc|^2.\\
\end{split}
\end{equation}

Applying the maximum principle to the evolution equation for $\Sc$
above, we obtain the following well known fact: if $\Sc(x,0) \geq  c$ then 
$\Sc(x,t)\geq c$ for all $x \in M$ and all $t\in [0,T)$.
In the following, we will assume (unless otherwise stated), that the
scalar curvature is bounded from below by $-1$ for all times: if it is not then
we may scale the solution $g(\cdot,t)$ by  $\ti g(\cdot,\ti t):= c g(\cdot, \frac {\ti
  t}{c})$,
where $c := |\inf_{x \in M} R(x,0)|>0$ to obtain
a new solution $(M,\ti g(\ti t))_{t \in [0,\ti T)}$, where $\ti T :=
c T$, which satisfies  $\inf_M \ti \Sc(x,\ti t ) \geq -1$ for all $0
\leq \ti t < \ti T$.

This means that the function  $ f := \frac{|\Rc|^2}{\Sc+2}$ is well
defined.

 In the proof of Lemma 3.2 of \cite{CaoX} the two evolution equations above were combined
to obtain the following evolution
equation for 
$ f := \frac{|\Rc|^2}{\Sc+c}$ at any point in space time where $\Sc +c
>0$ (see also Lemma \cite{Knopf}, where  related evolution
inequalities are shown).

\begin{equation}
\begin{split}
&\partt f = \lap f -2\frac{|Z|^2}
{(\Sc+c)^3} - 2 \frac{|\Rc|^4}{(\Sc+c)^2} + 4 \frac{\Rm(\Rc,\Rc)}{ (\Sc+c)}\\
\end{split}
\end{equation}
where $\Rm(\Rc,\Rc)$ is the function given locally by
$\Rm(\Rc,\Rc) = \Rm^{ikjl}\Rc_{ij}\Rc_{kl}$
and 
$Z:= (\grad \Ricci)(\Sc+c) - (\grad \Sc)( \Ricci)$ is the tensor
given locally by 
$Z_{kis}:= (\grad_k \Rc)_{is}(\Sc+c) - (\grad_k(\Sc+c)) (\Rc_{is})$.
The evolution equation for the integral of $f$ is then given as follows.

\begin{lemma}
Let $(M^n,g(t))_{t\in [0,T)}$ be a smooth solution to Ricci flow  on a
four dimensional 
compact manifold $M^n$ without boundary and assume that $\inf_M
\Sc(\cdot,0) >c$.  Then
\begin{eqnarray}
 &&\ddt \int_M  f d\mu_g = \int_M \Big(-2\frac{|Z |^2}{(\Sc+c)^3} 
 -2 f^2  +  4\frac {\Rm(\Rc,\Rc)}{(\Sc+c)}
- f \Sc \Big) \ \ d\mu_g \label{evolutionf}
\end{eqnarray}
for $ f := \frac{|\Rc|^2}{\Sc+c}$.
\end{lemma}
\begin{proof}
Use the above evolution equation for $f$ with the fact that
$\partt d\mu_g = -\Sc d\mu_g$ (see \cite{HaThree} for this last fact).
\end{proof}

As we mentioned above, there is no great loss of generality in
assuming $\inf_M \Sc(\cdot,t) \geq-1$ for all $t \in [0,T) $, and so we may choose $c=2$
(the special case that $\Sc >0$ everywhere, in which case we choose
$c=0$, will be handled below separately).
We will estimate the last two terms appearing on the right hand side of the
integral equality \eqref{evolutionf} using: (i) the Euler characteristic $\chi$,
(ii) the good (second) negative term on the right hand side of the above equality

\begin{theo}\label{firstestf}
Let $(M^4,g(t))_{t\in [0,T)}$ be a smooth solution to Ricci flow  on a
four dimensional compact manifold $M^4$ without boundary and assume that $\inf_M
\Sc(\cdot,0) > -1$.  Then
\begin{eqnarray}
 &&\ddt \int_M  f d\mu_g  \leq
2^8 \pi^2 \chi + \int_M ( -f^2 +   64 f + 2^{10} \Sc^2 )  \ \ d\mu_g
\label{generalint}
\end{eqnarray}
for $ f := \frac{|\Rc|^2}{\Sc+2}$.
\end{theo}

\begin{proof}
We estimate  the third term appearing on the right hand side of
\eqref{evolutionf}
with Young's inequality (remembering that we have now fixed $f = 
\frac{|\Rc|^2}{\Sc+2}$ in the definition of $f$ ):
\begin{equation}
\begin{split}
 4 \frac{\Rm(\Rc,\Rc)}{(\Sc+2)} 
& \leq  \frac{|\Rc|^4}{ 2(\Sc+2)^2} + 8|\Riem|^2  \\
& =   \frac{|\Rc|^4}{2 (\Sc+2)^2} +8 (|\Riem|^2 - 4 |\Rc|^2
+ \Sc^2) +  32|\Rc|^2 -8\Sc^2  \\
& = \frac{f^2}{2} + 8 I + 32|\Rc|^2 - 8\Sc^2 \label{Rmrc}
\end{split}
\end{equation}
where $I =  |\Riem|^2 -4 |\Rc|^2
+ \Sc^2$ is the integrand occurring in the generalised Gauss-Bonnet
Theorem, and $I$ satisfies
\begin{equation}
\int_M I \ \ d\mu_g = 2^5 \pi^2 \chi
\end{equation}
where $\chi = \chi(M)$ is the Euler characteristic of $M$  (see notes
in Appendix \ref{gaussbonnet}). Note that if
$M$ is not oriented, then this formula is correct with $\chi(M):=
\frac{1}{2}\chi(\ti M)$ where $\ti M$ is the {\it double cover} of $M$
,which is oriented (see Theorem 13.9  \cite{Lee}),
and $\chi(\ti M)$ is the Euler-characteristic of $\ti M$.
The second last term of the above inequality is
$ 32 |\Rc|^2 = 32 f (\Sc + 2)=32f \Sc + 64f  \leq \frac{f^2}{4} + 2^{10}
\Sc^2 + 64 f.$
Hence 

\begin{equation}
\begin{split}
 4 \frac{\Rm(\Rc,\Rc)}{(\Sc+2)}  \leq \frac{f^2}{2} + 8  I  + \frac{f^2}{4} +
 2^{10} \Sc^2 + 64f - 8\Sc^2
\end{split}
\end{equation}
Also,
\begin{equation}
-f  \Sc = \leq \frac{f^2}{4} + \Sc^2
\end{equation}
Combining these two estimates we obtain
\begin{eqnarray}
 -2 f^2  +  4\frac {\Rm(\Rc,\Rc)}{(\Sc+2)}
- f \Sc  &\leq& - f^2  +  8 I + 2^{10}  \Sc^2 + 64f.
\end{eqnarray}

Using this inequality in the equality \eqref{evolutionf} of the lemma
above, we get
\begin{eqnarray*}
\ddt \int_M  f \ \ d\mu_g &=&\int_M \Big( -2\frac{|Z |^2}{(\Sc+2)^3} 
 -2 f^2  +  4\frac {\Rm(\Rc,\Rc)}{(\Sc+2)}
- f \Sc \Big) \ \ d\mu_g \\
&\leq&\int_M ( -f^2 + 8 I +   2^{10} \Sc^2 + 64f) \ \ d\mu_g\cr
&=&2^8 \pi^2 \chi + \int_M ( -f^2 +   64 f + 2^{10} \Sc^2 )  \ \ d\mu_g
\end{eqnarray*}
as required.
\end{proof}

Integrating this inequality with respect to time, we obtain 
\begin{coro}\label{intestgen}
Let $(M^4,g(t))_{t\in [0,T)}$ be a smooth solution to Ricci flow  on a
compact four dimensional manifold $M^4$ without boundary and assume that $\inf_M
\Sc(\cdot,0) > -1$.  Then
\begin{eqnarray}
&& \int_M \frac{|\Rc|^2(\cdot,S)}{(\Sc(\cdot,S) +2)} d\mu_{g(S)}  +
 \int_0^S  \int_M  \frac{|\Rc|^4(\cdot,t)}{(\Sc(\cdot,t) +2)^2} d\mu_{g(t)} dt  \cr 
&& \ \ \ \ \leq 
2^2\pi^2 \chi(e^{64S} - 1) + e^{64S}\int_M
\frac{|\Rc|^2(\cdot,0)}{(\Sc(\cdot,0) +2)} d\mu_{g(0)} \cr
&& \ \ \ \ \ \
+2^{10}e^{64S}\int_0^S \int_M \Sc^2(\cdot,t) d\mu_{g(t)} dt   \cr
&& =: c_0(M,g(0),S) + 2^{10}e^{64S}\int_0^S \int_M \Sc^2(\cdot,t) d\mu_{g(t)} dt \label{generalint1}
 \end{eqnarray}
\begin{eqnarray}
 &&\int_M |\Rc|(\cdot,S) d\mu_{g(S)} \cr
&& \ \ \ \ \leq  \vol(M,g(S)) + 2c_0(M,g(0),S)  + 2^{11}  e^{64S}\int_0^S \int_M \Sc^2(\cdot,t) d\mu_{g(t)} dt  \label{generalint2}
\end{eqnarray}
and
\begin{eqnarray}
&&\int_0^S \int_M |\Rc|^2  d\mu_{g(t)} dt \cr
&&\ \ \ \ \leq \int_0^S \vol(M,g(t)) dt 
+2^3c_0(M,g(0),S)
+ 2^{13} e^{64S}\int_0^S \int_M \Sc^2(\cdot,t) d\mu_{g(t)} dt
 \label{generalint3}
\end{eqnarray}
and
\begin{eqnarray}
&&\int_0^S \int_M |\Rm|^2  d\mu_{g(t)} dt \cr 
&& \ \  \
\leq 4\int_0^S \vol(M,g(t)) dt 
+ 2^5( c_0(M,g(0),S) + \pi^2 \chi S) \cr
&& \ \ \ \ +2^{15} e^{64S}\int_0^S \int_M \Sc^2(\cdot,t) d\mu_{g(t)} dt 
 \label{generalint4} 
\end{eqnarray}
for all $0\leq S <T$, where $\chi = \chi(M)$ is the
Euler-characteristic of $M$, and 
\begin{eqnarray}
&& c_0(M,g(0),S):= 2^2\pi^2 \chi(e^{64S} - 1) + e^{64S}\int_M
\frac{|\Rc|^2(\cdot,0)}{(\Sc(\cdot,0) +2)} d\mu_{g(0)} 
\end{eqnarray}
is defined in the statement above.
\end{coro}

\begin{proof}
Using the inequality  \eqref{generalint}, we see that

\begin{eqnarray}
&&\ddt (e^{-64t}\int_M  f(\cdot,t) d\mu_{g(t)} ) + e^{-64t}\int_M
f^2(\cdot,t) d\mu_g(t) \cr
&& \ \ \ \ \ \leq
e^{-64t}2^8 \pi^2 \chi + e^{-64t}\int_M  2^{10} \Sc^2(\cdot,t)  d\mu_g(t)
\end{eqnarray}

Integrating this inequality from $0$ to $S$ implies
\begin{eqnarray}
&&e^{-64S}\int_M  f(\cdot,S) d\mu_{g(S)}  +  e^{-64S} \int_0^S\int_M
f^2(\cdot,t) d\mu_{g(t)} dt\cr
&&\ \ \ \  \leq e^{-64S}\int_M  f(\cdot,S) d\mu_{g(S)}  +  \int_0^S e^{-64t}\int_M
f^2(\cdot,t) d\mu_{g(t)} dt\cr
&& \ \ \ \  = \int_0^S \Big( \ddt (e^{-64t}\int_M  f(\cdot,t)
d\mu_{g(t)} )
+e^{-64t}\int_M f^2(\cdot,t) d\mu_{g(t)} \Big)dt \cr
&& \ \ \ \ \ \ + \int_M  f(\cdot,0) d\mu_{g(0)}  \cr
&& \ \ \ \ \leq
\int_{0}^S e^{-64t}2^8 \pi^2 \chi dt+\int_0^S e^{-64t}\int_M  2^{10}
\Sc^2(\cdot,t)  d\mu_g(t) dt + \int_M  f(\cdot,0) d\mu_{g(0)}  \cr
&& \ \ \ \ = -4(e^{-64S} - 1) \pi^2 \chi  +\int_0^S e^{-64t}\int_M  2^{10}
\Sc^2(\cdot,t)    d\mu_g(t) dt  + \int_M  f(\cdot,0) d\mu_{g(0)}   \cr
&& \ \ \ \ \leq 4(1-e^{-64S} ) \pi^2 \chi  +\int_0^S \int_M  2^{10}
\Sc^2 (\cdot,t)  d\mu_g(t) dt + \int_M  f(\cdot,0) d\mu_{g(0)} 
\end{eqnarray}
which, after multiplying by $e^{64S}$, gives us 
 the first integral inequality \eqref{generalint1}.

The second inequality can be obtained from the first as follows.
 \begin{eqnarray}
|\Rc| && \leq \frac{|\Rc|^2}{(\Sc +1)} + \frac{(\Sc + 1)}{4} \cr
&&\leq \frac{|\Rc|^2}{(\Sc +1)} + \frac{|\Sc|}{4} + \frac{1}{4} \cr
&& \leq  \frac{|\Rc|^2}{(\Sc +1)}  + \frac{|\Rc|}{2} +  \frac{1}{4} 
\end{eqnarray}
in view of the fact that (in four dimensions) $|\Sc| \leq 2 |\Rc|$,
and hence
\begin{eqnarray}
|\Rc|  && \leq  \frac{2|\Rc|^2}{(\Sc +1)}  +1.
\end{eqnarray}
The third inequality can be obtained from the first as follows.
\begin{eqnarray}
|\Rc|^2 && \leq \frac{4|\Rc|^4}{(\Sc +1)^2} + \frac{(\Sc + 1)^2}{16} \cr
&&\leq \frac{4|\Rc|^4}{(\Sc +1)^2} + \frac{|\Sc|^2}{8} + \frac{1}{8} \cr
&& \leq \frac{4|\Rc|^4}{(\Sc +1)^2} + \frac{|\Rc|^2}{2} + \frac{1}{8}
\end{eqnarray}
since 
$|\Sc|^2 \leq 4 |\Rc|^2 $ in four dimensions, and hence
\begin{eqnarray*}
&&|\Rc|^2  \leq \frac{8|\Rc|^4}{(\Sc + 1)^2} + 1.
\end{eqnarray*}
The last inequality follows from the third inequality in view of the
generalised Gauss-Bonnet theorem.
\end{proof}

In the case that $\Sc>0$ everywhere, we may consider the function
$f = \frac{|\Rc|^2}{\Sc}$ , that is we choose $c=0$. In this case,
some of the terms in the integral inequalities above simplify. In
particular, the volume terms don't appear.

\begin{theo}\label{posscalar}
Let $(M^4,g(t))_{t\in [0,T)}$ be a smooth solution to Ricci flow  on a
compact four dimensional manifold $M^4$ without boundary and assume that $\inf_M
\Sc(\cdot,0) > 0$.  Then
\begin{eqnarray}
 &&\int_M |\Rc|(\cdot,S) d\mu_{g(S)} \cr
&& \ \ \ \ \leq  2a_0(M,g(0),S)  + 2^{11}  e^{64S}\int_0^S \int_M \Sc^2(\cdot,t) d\mu_{g(t)} dt
 \label{generalintsc2}
\end{eqnarray}
and
\begin{eqnarray}
&&\int_0^S \int_M |\Rc|^2  d\mu_{g(t)} dt \cr
&&\ \ \ \ \leq 2^3a_0(M,g(0),S) 
+ 2^{13} e^{64S}\int_0^S \int_M \Sc^2(\cdot,t) d\mu_{g(t)} dt
 \label{generalintsc3}
\end{eqnarray}
and
\begin{eqnarray}
&&\int_0^S \int_M |\Rm|^2  d\mu_{g(t)} dt \cr 
&& \ \ \ \ 
\leq 2^5 \pi^2 \chi S + 2^5a_0(M,g(0),S)  + 2^{15}e^{64S}\int_0^S \int_M \Sc^2(\cdot,t) d\mu_{g(t)} dt
 \label{generalintsc4} 
\end{eqnarray}
for all $0\leq S <T$, where $\chi = \chi(M)$ is the
Euler-characteristic of $M$, and 
\begin{eqnarray}
&& a_0(M,g(0),S):= 2^2\pi^2 \chi(e^{64S} - 1) + e^{64S}\int_M
\frac{|\Rc|^2}{\Sc}(\cdot,0) d\mu_{g(0)}.
\end{eqnarray}
\end{theo}

\begin{proof}
We repeat the argument given in the proof of Theorem \ref{firstestf}, but
we use the function $f = \frac{|\Rc|^2}{\Sc}$ in place of
$f = \frac{|\Rc|^2}{\Sc + 2}$.
We use the fact that  $  |\Rc| = \frac{|\Rc|^2 }{|\Rc|} \leq
2\frac{|\Rc|^2}{\Sc} = 2f$ in four dimensions in the last part of the
argument to get
\eqref{generalintsc2} and \eqref{generalintsc3}.
The generalised Gauss-Bonnet theorem implies the last inequality \eqref{generalintsc4}.
\end{proof}

In the rest of this paper we often consider solutions $(M^4,g(t))_{t\in [0,T)}$ which satisfy the following {\it basic assumptions}.
\begin{itemize}
\item[(a)] $M^4$ is a smooth,  compact, connected four dimensional manifold without boundary
\item[(b)]  $(M^4,g(t))_{t \in [0,T)}$ is a smooth solution to the
  Ricci flow $\partt g(t) = -2\Ricci(g(t))$  for all $t \in [0,T)$ 
\item[(c)] $T< \infty$
\item[(d)] $\sup_{M^4\times[0,T)} |\Sc(x,t)| \leq 1$
\end{itemize}
If instead of $(d)$ we only have $\sup_{M\times[0,T)} |\Sc(x,t)| \leq
K < \infty$ for some constant $1<K<\infty$, then we may rescale the
solution $\ti g(\cdot,\ti t):= K g(\cdot, \frac {\ti t}{K})$ to obtain
a new solution $(M,\ti g(\ti t))_{t \in [0,\ti T)}$, where $\ti T := K T$, which satisfies the basic assumptions. 

Note that a solution which satisfies the basic assumptions also
satisfies
$\Sc(x,t) + 2 >0$ for all $ x \in   M$ for all $t\in [0,T)$ and hence
$ f := \frac{|\Rc|^2}{\Sc+2}$ is a well defined function.

For solutions satisfying the basic assumptions, a slight modification
of the above arguments leads to the following.
\begin{theo}\label{newfest}
Let $(M^4,g(t))_{t \in [0,T)}$ be a  solution to Ricci
flow satisfying the {\it basic assumptions}. 
Then 
\begin{eqnarray}
\ddt\int_M  f d\mu_g 
&\leq& 128\pi^2 \chi + \int_M ( -f^2 +   50 f) d\mu_g
\end{eqnarray}
for $f := \frac{|\Rc|^2}{\Sc+2}$, where $\chi = \chi(M)$ is the Euler
characteristic of $M$.
\end{theo}

\begin{proof}
Using almost the same argument given at the beginning of the proof of
Theorem \ref{firstestf}, we see that
\begin{equation}
\begin{split}
 4 \frac{\Rm(\Rc,\Rc)}{(\Sc+2)} 
& \leq  f^2 + 4 I + 16 |\Rc|^2,
\end{split}
\end{equation}
where $I =  |\Riem|^2 -4 |\Rc|^2
+ \Sc^2$ is the integrand occurring in the generalised Gauss-Bonnet
Theorem.
The last term of the above inequality is
$ 16 |\Rc|^2 = 16 f (\Sc + 2) \leq 48 f$ since $\Sc \leq 1$.
Hence 
\begin{equation}
\begin{split}
 4 \frac{\Rm(\Rc,\Rc)}{(\Sc+2)}  \leq f^2 + 4  I  + 48f
\end{split}
\end{equation}
Also,
\begin{equation}
-f  \Sc = -f  (\Sc +2) + 2f   \leq 2f.
\end{equation}
Combining these two estimates we obtain
\begin{eqnarray}
 -2 f^2  +  4\frac {\Rm(\Rc,\Rc)}{(\Sc+2)}
- f \Sc  &\leq& - f^2  +  4 I + 50f.
\end{eqnarray}

Using this inequality in the equality \eqref{evolutionf},  we get
\begin{eqnarray*}
\ddt \int_M  f \ \ d\mu_g &=&\int_M \Big( -2\frac{|Z |^2}{(\Sc+2)^3} 
 -2 f^2  +  4\frac {\Rm(\Rc,\Rc)}{(\Sc+2)}
- f \Sc \Big) \ \ d\mu_g \\
&\leq&\int_M ( -f^2 + 4 I +   50 f) \ \ d\mu_g\cr
&=&128 \pi^2 \chi + \int_M ( -f^2 +   50 f)  \ \ d\mu_g
\end{eqnarray*}
as required.
\end{proof}

\begin{theo}\label{integralest}
Let $(M^4,g(t))_{t \in [0,T)}$ be a  smooth solution to Ricci
flow satisfying the {\it basic assumptions}. Then we have the following estimates:
\begin{eqnarray} 
&& \int_M  |\Rc|^2 (\cdot,t)d\mu_{g(t)} \leq   b(g(0),t) \ \ \forall
\ t \in [0,T) \label{no1} \\
&& \int_M |\Riem|^2( \cdot,t) \ \ d\mu_{g(t)} \leq 32 \pi^2 \chi  +
4b(g(0),t) \ \ \forall
\ t \in [0,T)  \label{no2} \\
&& \int_{0}^t \int_M |\Rc|^4(\cdot,t) d\mu_{g(t)} dt\leq b(g(0),t)  \
\ \ \forall t \in [0,T]  \label{no3}\\
&&\int_S^T \int_M |\Rc|^p(\cdot,t) d\mu_{g(t)} dt \cr
&& \ \ \ \ \  \ \ \ \ \leq
(|b(g(0),T|)^{\frac p 4} e^{\frac {(4-p)T}{4}
 }(\vol(M,g(0))^{\frac {(4-p)}{4} } |T-S|^{\frac {(4-p)}{4} }  \to 0
\ \mbox { as } S \upto T \label{no4} 
\end{eqnarray}
for all $0< p < 4$, where 
\begin{eqnarray}
 b(g(0),t) && := 50e^{50t}\int_M  |\Rc|^2(\cdot,0)
 d\mu_{g(0)}   + 128 \pi^2 \chi (  e^{50 t} -1 ) 
\label{thebs}
\end{eqnarray}
and $\chi $ is
the Euler characteristic.
\end{theo}
\begin{remark}
In the above inequalities, we may estimate $b(g(0),t)$ and
$|b(g(0),t)|$ by 
$b(g(0),t) \leq |b(g(0,t)| \leq c(g(0),T) := 
50e^{50T}\int_M  |\Rc|^2(\cdot,0)
 d\mu_{g(0)}   + 128 \pi^2 |\chi|  e^{50 T}$
 \end{remark}
\begin{remark}
Note that $b(h,s):= b(\ti h,s)$ if $\ti h = ch$ and  $c>0, s>0$ are arbitrary, in view of
the fact that 
$\int_{M} |\Rc(h)|^2 d\mu_{h} = \int_M|\ti \Rc(\ti h)|^2 d\mu_{\ti h}$
in four dimensions, as one readily verifies using the definition of Riemannian
curvature (see for example the definition given in Section 2 of \cite{HaThree}). 
\end{remark}
\begin{remark}
In a recent paper, \cite{BZ}, the authors independently
showed (among other things) that 
$\sup_{t \in [0,T)} \int_M |\Rc|^2(\cdot,t) d\mu_{g(t)} < \infty $ and
\\$\sup_{t \in [0,T)} \int_M |\Rm|^2(\cdot,t) d\mu_{g(t)} < \infty$ if
$(M^4,g(t))_{t\in [0,T)}$ is a solution to Ricci flow satisfying the  basic assumptions.
Their method uses estimates on the heat kernel, which are also proved in
their paper, and their method is different from the method presented here.
\end{remark}

\begin{proof}

From Theorem \eqref{newfest}  above we have
\begin{eqnarray*}
\ddt \int_M  f d\mu_g &&\leq 128 \pi^2 \chi + \int_M ( -f^2 +   50 f)
\ \ d\mu_g
\end{eqnarray*}
and hence
\begin{eqnarray}
\dds \Big(e^{-50s} \int_M  f(\cdot,s)  \ \ d\mu_{g(s)}\Big) 
&&\leq e^{-50s}128 \pi^2 \chi  -e^{-50s}\int_M  f^2  \ \ 
d\mu_{g(s)} \label{intest1}
\end{eqnarray}

Integrating both sides of this inequality  in time from $0$ to $t<T$, we get
\begin{eqnarray}
&& e^{-50t} \int_0^t \int_M  f^2 d\mu_{g(s)} ds + e^{-50t}\int_M  f d\mu_{g(t)} \cr
 &&  \ \ \leq   \int_M  f(\cdot,0) d\mu_g(0)   + \frac{128}{50} \pi^2 \chi
 (1- e^{-50t}). 
 \label{defnhata} 
\end{eqnarray}
Using the definition of $f$ and the fact that $1\leq (\Sc + 2) \leq 3$ we see that
\begin{eqnarray} 
\frac { |\Rc|^2 }{3} \leq f = \frac{|\Rc|^2}{\Sc + 2} \leq  |\Rc|^2 \label{riccitof}
\end{eqnarray}
and hence,  using this in \eqref{defnhata}, we see that
\begin{eqnarray}
&& \frac{e^{-50t}}{50}  \int_0^t\int_M  |\Rc|^4 d\mu_{g(s)} ds  +
\frac{e^{-50t}}{50} \int_M  |\Rc|^2 d\mu_{g(t)} \cr 
&& \ \ \ \leq e^{-50t}\int_0^t\int_M  f^2 d\mu_{g(s)} ds +
e^{-50t}\int_M  f d\mu_{g(t)}  \cr
 &&  \ \ \leq  \int_M  f(\cdot,0) d\mu_g(0)   + \frac{128}{50} \pi^2
 \chi (1- e^{-50t})\cr
&&\leq \int_M |\Rc|^2(\cdot,0) d\mu_{g(0)} + \frac{128}{50} \pi^2 \chi (1- e^{-50t})
\end{eqnarray}
This gives us the first \eqref{no1} and third \eqref{no3}  estimate.

The second inequality, \eqref{no2}, follows from the first
inequality and  the generalised Gauss-Bonnet Theorem:
\begin{equation}
\int_M |\Riem|^2 \ \ d\mu_g = 32 \pi^2 \chi + \int_M ( 4 |\Rc|^2
- \Sc^2) d\mu_g \leq  32 \pi^2 \chi  + 4 b(g(0),t),
\end{equation}
as required.

The equation for the evolution of the volume is (see Section 3 of \cite{HaThree})
 $\ddt \vol(M,g(t)) = -\int_M \Sc d\mu_{g(t)}$, and hence using
 $|\Sc|\leq 1$ we see that 
$ -\vol(M,g(t)) \leq  \ddt \vol(M,g(t))  \leq \vol(M,g(t))$. 
Integrating this inequality from $0$ to $t$ we see that 
$e^{-T}\vol(M,g(0)) \leq \vol(M,g(t)) \leq e^{T}\vol(M,g(0))$.
Using H\"older's inequality and these  volume bounds we get for $p < 4$ and $S < R < T$
\begin{eqnarray} 
\int_S^R \int_M |\Rc|^p(\cdot,l) d\mu_{g(l)} dl && \leq  \Big(\int_S^R \int_M
|\Rc|^4 d\mu_{g(l)} dl\Big)^{p/4} \Big(\int_{S}^R \int_M d\mu_{g(l)} dl\Big)^{1/q} \cr
&& \leq    \Big(\int_0^T \int_M
|\Rc|^4 d\mu_{g(l)} dl\Big)^{p/4} |S-R|^{\frac 1 q} e^{\frac T
  q}\Big(\vol(M,g(0)\Big)^{\frac 1 q}\cr
&& \leq |b(g(0),T)|^{\frac p 4} e^{\frac T
  q}(\vol(M,g(0))^{\frac 1 q} |S-R|^{\frac 1 q} 
\end{eqnarray}
where $ \frac 1 q = 1- \frac p 4 = \frac {(4-p)}{4}.$
This implies the fourth inequality, \eqref{no4} above.

This completes the proof.
\end{proof}

In four dimensions, $\int_M |\Rc|^2 d\mu_g  $ and $\int_M |\Riem|^2
d\mu_g$ are scale invariant quantities: if $\ti g = cg$, $c>0$, then
$\int_M |\ti \Rc|^2 d\mu_{\ti g} = \int_M |\Rc|^2 d\mu_g  $ and 
$\int_M |\Rm|^2
d\mu_g = \int_M |\ti \Rm|^2 d\mu_{\ti g}  $, as can be readily
verified using the definition of Riemannian curvature, as we mentioned
above.

These facts help us to obtain inequalities for scaled solutions, as
explained below in the proof of the following theorem.

\begin{theo}
Let $(M^4,g(t))_{t \in [0,T)}$ be a  smooth solution to Ricci
flow satisfying the basic assumptions.
Let $\ti g(\cdot,\ti t):= cg(\cdot, \frac{\ti  t}{ c})$ for $0 \leq \ti t \leq
\ti T:= cT$ for any constant $c>0$, and let $0\leq  \ti R < \ti S\leq \ti
T.$
Then 
\begin{eqnarray} 
&&\int_M  |\ti \Rc|^2 (\cdot,\ti t)d\mu_{g(\ti t)}   \leq  b(g(0),t) \label{scaled1}\\
&&\int_M |\ti \Riem|^2( \cdot,\ti t) \ \ d\mu_{g(\ti t)}   \leq 32 \pi^2 \chi  +
4 b(g(0),t) \label{scaled2}\end{eqnarray}
for all $\ti t \in [0,\ti T]$, where $t:= \frac{\ti t}{c}$.
In the case that we additionally assume $c \geq 1$, then we also have
\begin{eqnarray} 
&&\int_{\ti R}^{\ti S} \int_M |\ti \Rc|^4(\cdot,\ti l)  d\mu_{g(\ti l)} d\ti l  \leq 50e^{50\ti L} b(g(0),R) +128  \pi^2 \chi
 (e^{50\ti L}- 1)  \cr
\label{scaled3}
\end{eqnarray}
where $\ti L = \ti S- \ti R$, $b(g(0),t)$ 
is defined above in \eqref{thebs} and $\ti t = ct$, $R:= \frac{\ti R}{c}$, 
$ S = \frac{\ti S}{c}$.
\end{theo}
\begin{remark}
Note that $b(g(0),t) $ is {\bf not} equal to    $
b(\ti g(0), \ti t)$ ) in general: for fixed $t \in [0,T)$,  the quantity $ b(\ti g(0), \ti t)
\to \infty$ as $c \to \infty$ for $\ti t:= ct$,  if for example $\chi
>0$.
As we mentioned above, we do have however $b(g(0),t) = b(\ti g(0),t) $
for all $t >0$. 
\end{remark}

\begin{proof}

The first two inequalities follow from the fact that the left hand
side of the inequality is scale invariant (see the explanation given
just before the statement of this theorem). Now we consider the case
that $ c \geq 1$. Define 
$(M,h(t))_{ t\in [0,\ti L:= \ti S-\ti R)}$ to be 
$h(s):= \ti g(\cdot, \ti R+s)$. Then  $(M,h(s))_{ s\in [0,\ti L)}$ is a
solution satisfying the basic assumptions, and so we may apply the
results above to obtain
\begin{eqnarray*}
\int_{\ti R}^{\ti S} \int_M |\ti \Rc|^4(\cdot,\ti t) d\mu_{\ti g (\ti t)}
d\ti t 
&& = \int_{0}^{\ti L} \int_M |\Rc|^4(\cdot,t) d\mu_{h(t)} dt \cr
&&\leq b(h(0),\ti L)\cr
&&= 50 e^{50\ti L}
\int_M  |\Rc|^2(\cdot,0) d\mu_h(0)  + 128  \pi^2 \chi
 (e^{50\ti L}-1) \cr
&& = 50 e^{50\ti L}   \int_M |\ti \Rc|^2(\cdot,\ti R) d\mu_{\ti g (\ti R)} +
 128  \pi^2 \chi
 (e^{50\ti L}- 1) \cr
&& \leq 50 e^{50\ti L} b(g(0),R) +128  \pi^2 \chi
 (e^{50\ti L}- 1) 
\end{eqnarray*}
where the last  inequality here follows in view of the first (scale invariant)
inequality  \eqref{scaled1}.
This finishes the proof.
\end{proof}

The following  corollaries are obtained using the above integral
estimates.
%Although the inequalities therein are not used in this paper, they
%may be of independent interest.

\begin{coro}\label{gradientestimate}
Let $(M^4,g(t))_{t \in [0,T)}$ be a  smooth solution to Ricci
flow on a closed four manifold $M$ satisfying the basic estimates with
$T< \infty$.
Then \begin{eqnarray}
 && \int_{0}^T \int_M |\grad \Rc|^2 d \mu_{g(t)} dt \cr
&& \ \ \ \leq B(g(0),T)\cr
&& \ \ \ := 
\int_M |\Rc(\cdot,0)|^2  d \mu_{g(0)}  +b(g(0),T) +2^9
\pi^2\chi T \cr
&& \ \ \ \ \ + 2^6\Big( 
 (e^{50T}-1) \int_M |\Rc(\cdot,0)|^2  d
\mu_{g(0)}   + 128 \pi^2 \chi( \frac{e^{50T}}{50} - \frac{1}{50} - T)\Big)
\label{gradienteq}
\end{eqnarray}
\end{coro}
\begin{remark}
Note that $B(h,s) = B(ch,s)$ for all $c>0$, for all $s>0$, since this
is true for $b(h,s)$, and $\int_M |\Rc|^2 d\mu_g$ is invariant
under scaling (as we explained above).
\end{remark}
\begin{proof}
As mentioned above, see \eqref{riccievn},
the evolution equation for $|\Rc|^2$ is 
$$\partt |\Rc|^2 = \lap |\Rc|^2 - 2 | \grad \Rc|^2 + 4\Rm(\Rc,\Rc).$$
Integrating this over space  and time from $0$ to $T$  we get
\begin{eqnarray} 
&& \int_{0}^T \int_M |\grad \Rc|^2( \cdot,t) \mu_{g(t)} dt \cr
&& \ \ \  \leq \int_M |\Rc|^2(\cdot,0)  d \mu_{g(0)}  + \int_0^T \int_M 4 |\Rm(\Rc,\Rc) |(\cdot,t)d \mu_{g(t)} dt \cr
&& \ \ \ \leq \int_M |\Rc|^2(\cdot,0) d \mu_{g(0)}  + \int_0^T \int_M |\Rc|^4(\cdot,t) d
\mu_{g(t)} dt  \cr
&& \ \ \ \ \
+ \int_0^T \int_M 16 |\Rm|^2(\cdot,t) d \mu_{g(t)} dt 
\end{eqnarray}
Now we use the  inequalities  \eqref{no2} and \eqref{no3} to get
\begin{eqnarray} 
&& \int_{0}^T \int_M |\grad \Rc|^2 d \mu_{g(t)} dt \cr
&& \leq  \int_M |\Rc(\cdot,0)|^2  d \mu_{g(0)} + b(g(0),T) + 16 \int_0^T( 32
\pi^2 \chi +4 b(g(0),t)) dt\cr
&& = \int_M |\Rc(\cdot,0)|^2  d \mu_{g(0)}  +b(g(0),T) +16 \cdot 32
\pi^2\chi T + 16\cdot 4 \cdot \int_0^T b_t dt \cr
&& = \int_M |\Rc(\cdot,0)|^2  d \mu_{g(0)}  +b(g(0),T) +2^9
\pi^2\chi T \cr
&& \ \ + 2^6\Big( 
 (e^{50T}-1) \int_M |\Rc(\cdot,0)|^2  d
\mu_{g(0)}   + 128 \pi^2 \chi( \frac{e^{50T}}{50} - \frac{1}{50} - T)
\Big)\cr
&& =: B(g(0),T) 
\end{eqnarray}
as required
\end{proof}

\begin{coro}
Let $(M^4,g(t))_{t \in [0,T)}$ be a  smooth solution to Ricci
flow on a closed four manifold $M$ with
$T< \infty$.
Let $K:= \sup_{M \times [0,T)} |\Sc| < \infty$ and assume $K \geq 1$.
Then we have the following estimates:
\begin{eqnarray} 
&& \int_M  |\Rc|^2 (\cdot,t)d\mu_{g(t)} \leq  b(g(0),Kt)  \ \ \forall
\ t \in [0,T) \label{no1g} \\
&& \int_M |\Riem|^2( \cdot,t) \ \ d\mu_{g(t)} \leq 32 \pi^2 \chi  +
4b(g(0),Kt) \cr 
\ \ \forall
\ t \in [0,T)  \label{no2g} \\
&& \int_{0}^T \int_M |\Rc|^4(\cdot,t) d\mu_{g(t)} dt\leq Kb(g(0),KT)  \label{no3g}\\
&&\int_{0}^T \int_M |\grad \Rc|^2 d \mu_{g(t)} dt \leq   K^3B(g(0),KT) \label{gradientg}
\end{eqnarray}
where  $b(g(0),s), B(g(0),s)$ are as defined in
\eqref{thebs} and \eqref{gradienteq}
($s \in [0,\infty)$ is arbitrary in the definition).

\end{coro}
\begin{remark} If $K \leq 1$, then we may estimate the integrals
  involved using Theorem \ref{integralest}
and Corollary \ref{gradientestimate}
\end{remark}

\begin{proof}
Set $\ti g(\cdot,\ti t):= K g(\cdot, \frac{\ti t}{K})$
for $\ti t \in [0,T K =: \ti T]$. The result now follows from  Theorem
\ref{integralest} and Corollary  \ref{gradientestimate} applied to the solution  $\ti g(\cdot,\ti t)_{t \in
  [0,\ti T)}$ in view of the fact that $\ti T =  T K $ and the identities
$\int_M  |\Rc|^2(\cdot,t) d\mu_{g(t)} 
= \int_M  |\ti \Rc|^2(\cdot,\ti t) d\mu_{\ti g(\ti t)} $, 
$\int_M  |\Rm|^2(\cdot,t) d\mu_{g(t)} 
= \int_M  |\ti \Rm|^2(\cdot,\ti t) d\mu_{\ti g(\ti t)} $, \\
$ \int_0^T \int_M |\Rc|^4 d\mu_{g(t)} dt = K \int_0^{\ti T} \int_M
|\ti \Rc|^4 d\mu_{\ti g(\ti t) } d\ti t$, \\ $ \int_0^T \int_M |\grad \Rc|^4 d\mu_{g(t)} dt = K^3 \int_0^{\ti T} \int_M
|\ti \grad \ti \Rc|^4 d\mu_{\ti g(\ti t) } d\ti t$,
which all follow from the scaling,
that is $\ti g(\cdot,\ti t) = K g(\cdot,\frac{\ti t}{K})$, and the
fact that $B(g(0),s) = B(\ti g(0),s)$, $b(g(0),s) = b(\ti g(0),s)$ for
all $s>0$, as we mentioned above.

\end{proof}
\begin{appendix}
\section{Notes on the Euler characteristic}\label{gaussbonnet}
In the following, we assume that $(M^n,g)$ is an oriented smooth compact
Riemannian manifold
without boundary (unless otherwise stated).
The Pfaffian is a second order polynomial obtained using the curvature
operator: see section 3 in \cite{BG}, for example, for a
definition. In the case that the Riemannian manifold $(M^4,g)$ we
consider has dimension four, the Pfaffian $\Pf$ may be written as 
$\Pf = c(|\Riem|^2 -4 |\Rc|^2 + \Sc^2 )\vol_g $ where $\vol_g$ is the
Riemannian volume form. This may be seen by using the Formula 4.1
and Corollary 4.1 in \cite{BG}  ( using the orthonormal basis
$X_1,X_2,X_3,X_4$ given at $p \in M$ coming from Corollary 4.1 in
\cite{BG} we can calculate $(|\Riem|^2 -4 |\Rc|^2
+ \Sc^2) \vol_g$  and we see that it has the same value ( up to a constant)
of $\Pf$ at $p$ given in the Formula 4.1 of \cite{BG} ).
It is known (\cite{AW}, \cite{Chern}), 
that the generalised Gauss-Bonnet formula $c(n) \int_M \Pf = 
\euler(M)$ holds: for an intrinsic explanation using  modern day
terminology, see \cite{Br}. One may choose the generic vector field
$Y$ occurring in the explanation of Bryant to be $\grad f$, where $f:
M \to \R$ is a morse-function (see section 6 in \cite{Mi} to see that
such functions exist). The Euler characteristic is 
$\euler(M):= c_0 -c_1 + c_2-c_3 + \ldots +(-1)^{n}c_n = b_0
-b_1+b_2-b_3 + \ldots + (-1)^n b_n$ , where here, $b_i$ is the $i$th
Betti-number, and $c_i$ is the number of critical values of degree $i$
: see Theorem 5.2 in \cite{Mi}.
So we have $ a \int_M  (|\Riem|^2 -4 |\Rc|^2 + \Sc^2 ) \vol_g=
\euler(M)$ for some constant $a$. To see that the constant $a$ in this
formula is $\frac{1}{32 \pi^2}$,
calculate  the left and right hand side of this formula in the case that $(M,g)$ is
the standard sphere with sectional curvature equal to one everywhere.
\end{appendix}

\end{document}